\documentclass{amsart}
\usepackage{amsaddr}
\usepackage{amsmath}
\usepackage{epic,eepic,latexsym, amssymb, amscd, amsfonts, xypic}
\usepackage{color}
\usepackage{bbm}
\usepackage{enumerate}
\usepackage{setspace}

\newcommand{\comments}[1]{}

\usepackage{hyperref}

 \newlength{\baseunit}               
 \newcount{\numlines}                
\newtheorem{remk}{Remark}
\newtheorem{thm}{Theorem}
\newtheorem{lem}{Lemma}

\newtheorem{prop}{Proposition}

\newcommand{\oleq}[1]{\overset{ {\scriptscriptstyle #1}}{\leq}}

\def \N {\mathbb N}

\def \cF {\mathcal F}

\def \cA {\mathcal A}
\def \c\nu {\mathcal \nu}
\def \cW {\mathcal W}

\def \cH {\mathcal H}
\def \cI {\mathcal I}

\def \cJ {\mathcal J}

\def \cH {\mathcal H}

\def \o\nu {{\overline{\nu}}}

\def \h\nu {{\widehat{\nu}}}
\def \up\nu {\nu^{\up}}

\def \path {\mbox{\textbf{path}}}

\def \up {{\text{\tiny up}}}
\def \low {{\text{\tiny low}}}
\def \med {{\text{\tiny med}}}
\def \high {{\text{\tiny high}}}

\newcommand{\dist}{\operatorname{dist}}

\newcommand{\supp}{\operatorname{supp}}
\newcommand{\spa}{\operatorname{span}}
\newcommand{\diam}{\operatorname{diam}}

\newcommand{\bump}{a}

\newcommand{\Z}{\mathbb{Z}}

\newcommand{\C}{\mathbb{C}}

\newcommand{\R}{\mathbb{R}}

\comments{
\numberwithin{equation}{section}
\numberwithin{thm}{section}
\numberwithin{defn}{section}
\numberwithin{lem}{section}
\numberwithin{cor}{section}
\numberwithin{prop}{section}
\numberwithin{alg}{section}
}

\begin{document}
\pagestyle{plain}
\title{The Eigenvalue Distribution Of Time-Frequency Localization Operators}

\author{Arie Israel}
\address{Mathematics Department, University of Texas at Austin.}

\maketitle

\begin{abstract}
We estimate the distribution of the eigenvalues of a family of time-frequency localization operators whose eigenfunctions are the well-known Prolate  Spheroidal Wave Functions from mathematical physics. These operators are fundamental to the theory of bandlimited functions and have applications in signal processing. Unlike previous approaches which rely on complicated formulas for the eigenvalues, our approach is simple: We build an orthonormal basis of modulated bump functions (known as wave packets in time-frequency analysis) which approximately diagonalizes the operator of interest.
\comments{
Time-frequency localization operators (TFLOs) are of interest as they model the frequency-and time-limited nature of many real-world signals, and the behavior of their eigenvalues is fundamental to their understanding.  Theoretical results in the literature have been largely limited to the one-dimensional setting. Here we extend one of the principal results on the behavior of TFLOs to higher dimensions: they have a number of `large' eigenvalues, then a small number of intermediate eigenvalues, and the remaining eigenvalues are near zero. The number of large eigenvalues, termed the numerical rank, is proportional to the product of the Lebesgue measures of the time and frequency intervals.
}
\end{abstract}

\section{Introduction}
\label{sec0}

This paper concerns the distribution of eigenvalues of \emph{time-frequency localization operators} (TFLOs) of the form
\begin{equation}\label{tflo}
T_{I, J} f = R_I P_J R_I  f
\end{equation}
where $I$, $J$ are compact intervals, and $R_I : L^2(\R) \rightarrow L^2(\R)$, $P_J : L^2(\R) \rightarrow L^2(\R)$ are associated projection operators in the time and frequency variables, i.e.
\begin{align}\label{tflo_a}
&(R_I f)(x) = \left\{ 
\begin{array}{ll}
 f(x) &  x \in I  \\
 0 & x \in \R \setminus I,
\end{array}
\right.\\
\label{tflo_b}
&(P_J f) (x) = \int_{J} \int_{\R} f(y) e^{- 2 \pi i \omega \cdot (y - x)} dy d \omega.
\end{align}

The eigenfunctions of $T_{I,J}$ are the restrictions to $I$ of Prolate Spheroidal Wave Functions (PSWFs),  which arose first in the study of the Helmholtz equation in mathematical physics \cite{St}. 
Due to the connection with TFLOs, the PSWFs provide an optimal basis for the representation of bandlimited functions on an interval, as observed in Landau-Pollack \cite{LP1,LP2}, Slepian \cite{S1,S2}, and Slepian-Pollack \cite{SP1}. More recently, PSWFs have been applied to produce quadrature formulas \cite{OR1} and interpolation schemes \cite{OR2} for computing with bandlimited functions.

As demonstrated in \cite{L1}, the eigenvalues of $T_{I,J}$ display a concentration phenomenon in the asymptotic limit as the parameter $\Lambda:= \lvert I \rvert \cdot \lvert J \rvert$ tends to infinity: For $\epsilon > 0$, there are approximately $\Lambda$ eigenvalues in the interval $[1,1-\epsilon]$, approximately $C \log (\Lambda) \cdot \log(1/\epsilon)$ eigenvalues in the interval $( \epsilon, 1 - \epsilon )$, and the remaining eigenvalues form a sequence tending to zero at an exponential rate. This estimate is asymptotic, meaning that it is guaranteed only for some sufficiently large $\Lambda$; no quantitative bounds were known for any finite $\Lambda$ until recently: A quantitative upper bound on the eigenvalue sequence was proven in Osipov \cite{O1}.

This paper's aim is to present an alternate method for estimating the eigenvalues of a TFLO  using techniques from time-frequency analysis. We will prove quantitative upper and lower bounds by giving an example of a basis of time-frequency wave packets (the local-cosine basis) that approximately diagonalizes the TFLO.
Our results are sub-optimal by a single factor of $\log(1/\epsilon)$. 
We hope that the methods given here can be used to study localization operators associated to domains in higher-dimensions. We leave this investigation for a future work. 

\subsection{Statement of main results}\label{results_sec}

We fix intervals $J = [-1/2,1/2]$ and $I = [-D/2,D/2]$, where $D \geq 2$.

The operator $T = T_{I,J}$ defined in \eqref{tflo} is compact and positive-semidefinite, and its $L^2$ operator norm is at most $1$. We write $L^2(\R)$ as the direct sum $\mathcal{H} \oplus \ker (T)$, where $\cH = \ker(T)^\perp$; clearly, $L^2(\R \setminus I)$ is a subspace of $\ker(T)$. Thus, we may regard $T$ as an operator on $L^2(I)$. The spectral theorem for compact operators implies that the spectrum of $T$ is discrete and the positive eigenvalues of $T$ form a non-increasing sequence $\{ \lambda_k\}$ in $(0,1]$, with $\lambda_k \rightarrow 0$ as $k \rightarrow \infty$. Furthermore, there exists an orthonormal basis $\{ \psi_k \}_{k \in \N}$ for the subspace $\cH = \ker(T)^\perp$ of $L^2(I)$ satisfying the eigenvalue equation $T \psi_k =\lambda_k \psi_k $ for all $k \in \N$.

The position of the eigenvalue of $T$ closest to $\lambda = 1/2$ is described in the next result from \cite{L2}. 

\begin{thm}\label{thm0}
We have
\begin{equation}
\label{pos_mid}
\lambda_{[D] - 1} \leq \frac{1}{2} \leq \lambda_{[D]}.
\end{equation}
Here, we write $[x]$ to denote the integer part of a real number $x$.
\end{thm}

Our main result bounds the number of eigenvalues contained in an interval centered about $\lambda = 1/2$.

\begin{thm}\label{thm1} For each $\eta \in (0,1/2]$ there exists a constant $A_\eta \geq 1$ such that the following holds.

Given $\epsilon \in (0,1/2)$ and $D \geq 2$,  define
\begin{equation*}
K = A_\eta \cdot \left( \log \left( \log (D)  \cdot \epsilon^{-1} \right) \right)^{1 + \eta}  \cdot \log(D \cdot \epsilon^{-1}).
\end{equation*}
Then
\[
\bigl\{ k \in \N :   \lambda_k \in ( \epsilon,1-\epsilon)  \bigr\} \subset [D - K , D + K].
\]
\end{thm}
\begin{remk}
In \cite{O1}, a one-sided bound $\bigl\{ k  \in \N : \lambda_k \geq \epsilon \bigr\} \subset [0, D + K']$ is proven, with $K' = C \log( D )^2 \log(1/\epsilon)$. Our result  has an improved dependence on $D$, but sub-optimal dependence on $\epsilon$. The optimal bound is expected to be $K'' = C  \log(D) \log(1/\epsilon)$, as predicted by the asymptotic analysis in \cite{L1}.
\end{remk}

Our strategy for proving Theorem \ref{thm1} will be to construct an \emph{approximate eigenbasis} for the operator $T$.  We define a family of translated and modulated bump functions (the local-cosine basis). The result follows by counting the number of basis functions whose ``time-frequency profile'' is localized inside the rectangle $I \times J$ in phase space.

The outline of the paper is as follows. In Section \ref{sec1} we present the functional analysis lemma that will be used to control the eigenvalues of our operator. In Section \ref{sec2} we construct a local-cosine basis and present its basic properties. We prove the key energy estimates for this  basis in Section \ref{sec3}. We conclude the paper in Section \ref{sec4} with the proof of Theorem \ref{thm1}.

\subsection{Acknowledgements}

We are grateful to Mark Tygert for useful discussions and for his comments during the preparation of this work. We are also grateful to Rachel Ward for proofreading an early version of this paper.

\section{The Main Lemma}
\label{sec1}


Let $\mathcal{H}$ be a complex Hilbert space, and let $\{ \phi_k\}_{k \in \cI}$ be an orthonormal basis for $\mathcal{H}$. Denote the inner product on $\cH$ by $\langle \cdot, \cdot \rangle$ and the Hilbert norm on $\cH$ by $\| \varphi \| := \langle \varphi,\varphi \rangle$ for $\varphi \in \cH$. 

Let $T : \mathcal{H} \rightarrow \mathcal{H}$ be a compact, positive-semidefinite operator, with operator norm at most $1$. By the spectral theorem for compact operators, there is a sequence $1 \geq \lambda_1 \geq \lambda_2 \geq \cdots > 0$, and a complete orthonormal basis $ \{ \varphi_\ell\}_{\ell \in \N}$ for the subspace $\mathcal{H}' = \ker(T)^\perp$ of $\cH$, satisfying the eigenvalue equation $T \varphi_\ell = \lambda_\ell \cdot \varphi_\ell$ for all $\ell \in \N$. For $\epsilon \in (0,1/2)$, let $M_\epsilon = M_\epsilon(T)$ denote the number of eigenvalues $\lambda_\ell$ (counted with multiplicity) that belong to the interval $(\epsilon,1-\epsilon)$.

\begin{lem}\label{lem1} 

Let $\{\phi_k\}_{k \in \cI}$ be an orthonormal basis for $\cH$. Assume that $\cI$ is the disjoint union $\cI_0  \cup \cI_1 \cup \cI_2$, where
\begin{equation} \label{assump1} \sum_{k \in \cI_0} \| T \phi_k \|^2 + \sum_{k \in \cI_2} \| T \phi_k - \phi_k \|^2 \leq \epsilon^3.
\end{equation}
Then we have $M_\epsilon \leq 2 \cdot \# ( \cI_1)$.
\end{lem}

\begin{proof}

Note that $M_\epsilon$ is the dimension of the subspace $S_\epsilon := \spa \{ \varphi_\ell : \lambda_\ell \in (\epsilon,1- \epsilon) \}$ of $\cH$.

Consider the orthogonal projection operator $\pi_\epsilon : \cH \rightarrow S_\epsilon$. Note that $T$ and $\pi_\epsilon$ commute, since $S_\epsilon$ is spanned by a  collection of eigenvectors of $T$.

By definition, we have $\epsilon \| \phi \| \leq \| T \phi \|$ and $\epsilon \| \phi \| \leq \| T \phi - \phi \|$ for all $\phi \in S_\epsilon$. Hence, 
\begin{align*}
&\epsilon \cdot \| \pi_\epsilon \phi \| \leq  \| T \pi_\epsilon \phi\| = \| \pi_\epsilon T \phi \| \leq \| T \phi \|, \; \mbox{and} \\ 
& \epsilon \cdot \| \pi_\epsilon \phi \| \leq \| T \pi_\epsilon \phi - \pi_\epsilon \phi \|  = \| \pi_\epsilon( T \phi - \phi ) \| \leq \| T \phi - \phi\| \quad \mbox{for all} \; \phi \in \cH. \notag{}
\end{align*}
Using $\phi = \phi_k$ ($k \in \cI_0 \cup \cI_2$) in the above estimates, we see that
\[  \sum_{k \in \cI_0 \cup \cI_2} \epsilon^2 \cdot \| \pi_\epsilon \phi_k \|^2 \leq \sum_{k \in \cI_0 } \| T \phi_k \|^2 + \sum_{k \in \cI_2 } \| T \phi_k - \phi_k \|^2 \oleq{\eqref{assump1}} \epsilon^3 .\]
Therefore,
\begin{equation} \label{e1a}
\sum_{k \in \cI_0 \cup \cI_2} \| \pi_\epsilon \phi_k \|^2 \leq \epsilon \leq 1/2.
\end{equation}

We may assume that $M_\epsilon = \dim(S_\epsilon) \geq 1$, for otherwise the conclusion of the lemma is trivial. Then, by the Parseval identity and the fact that $\pi_\epsilon$ is an orthogonal projection operator for $S_\epsilon$, we have
\begin{equation*}
\| \psi \|^2 = \sum_{k \in \cI}  \langle \psi, \phi_k \rangle^2  = \sum_{k \in \cI} \langle \psi , \pi_\epsilon \phi_k \rangle^2 \quad \mbox{for all} \;  \psi \in S_\epsilon.
\end{equation*}
We fix an orthonormal basis for $ S_\epsilon$ and sum the previous estimate over all $\psi$ belonging to that basis. Thus we obtain
\begin{equation}\label{newstuff}
\dim(S_\epsilon) = \sum_{k \in \cI} \| \pi_\epsilon \phi_k \|^2,
\end{equation}
where here again we have used the Parseval identity to simplify the right-hand side. Recall that $\cI$ is the disjoint union $\cI_0 \cup \cI_1 \cup \cI_2$. Since $\dim(S_\epsilon) \geq 1$, we learn from \eqref{e1a} and \eqref{newstuff} that
\begin{equation}
\label{newstuff1} 
\sum_{k \in \cI_1 } \|  \pi_\epsilon \phi_k \|^2 \geq \dim(S_\epsilon) - 1/2 \geq (1/2) \cdot \dim(S_\epsilon).
\end{equation}
Finally, note that $\| \pi_\epsilon \phi_k \|^2 \leq \| \phi_k \|^2 = 1$ for any $k \in \cI_1$. Thus \eqref{newstuff1} yields $M_\epsilon = \dim(S_\epsilon) \leq 2  \cdot \#(\cI_1)$, as desired. \end{proof}

\section{Local Trigonometric Bases}\label{sec_trig}
\label{sec2}

In this section we exhibit a smooth compactly supported cutoff function whose Fourier transform has near-exponential decay. We follow an approach found in \cite{DH}. Using this cutoff function we construct an orthonormal basis for $L^2(I)$ ($I$ a compact interval) consisting of modulated bump functions. This is the \emph{local cosine basis} of Coifman-Meyer \cite{CM}.

\subsection{A cutoff function}\label{cut_sec}

Fix an integer $m \geq 1$, and define
\begin{equation} \label{c0}
\bump(x) = \left\{
        \begin{array}{ll}
            e^{- (1-x)^{-m}} e^{ - (x+1)^{-m}}  & \quad x \in (-1,1)\\
            0   & \quad x \in (-\infty,1] \cup [1,\infty).
        \end{array}
    \right.
 \end{equation}
Write  $D^k f$ to denote the $k$-fold derivative of a function $f : \R \rightarrow \R$.
 
\begin{lem}\label{der_bounds1} 
We have
\[ \lvert D^k \bump(x) \rvert \leq (16m)^{k} \cdot k^{(1+m^{-1})k} \;\;\; \mbox{for} \; k \geq 0, \; x \in \R.
\]
Here our notation is that $0^0 = 1$.
\end{lem}

\begin{proof}
First observe that $D^k (e^{-x^{-m}})$ ($x>0$) is equal to the sum of $2^k$ terms of the form 
\begin{align*}
F_{w,j,r}(x) := w \cdot x^{-[ (m+1)j + r]} \cdot e^{-x^{-m}},
\end{align*}
where $w$ is a real number, and $j,r$ are integers satisfying $j + r = k$ and $\lvert w \rvert \leq m^j \cdot [(m+1)j + r]^r$. This statement is easily proven by induction on $k$.

Using the estimate $y^{R} e^{- y} \leq R^R$, for $y,R > 0$, we obtain
\begin{align*}
\lvert F_{w,j,r}(x) \rvert &\leq m^j \cdot [(m+1)j + r]^r \cdot \left[ [(m+1)j + r] \cdot m^{-1} \right]^{[(m+1)j + r] \cdot m^{-1}} \\
& = m^{j+r} \cdot \left[ [ (m+1)j + r] \cdot m^{-1} \right]^{(m+1)\cdot(j + r) \cdot m^{-1}} \\
& \leq m^k \left[(1 + m^{-1}) \cdot k \right]^{(1+m^{-1}) \cdot k} \\
& \leq (4m)^k  \cdot k^{(1+m^{-1}) \cdot k}.
\end{align*}
We conclude that $\lvert D^k (e^{-x^{-m}}) \rvert \leq (8m)^k k^{(1+m^{-1})\cdot k}$ for $x >  0$. Hence, by the Leibniz rule we have
\begin{align*}
 \lvert D^k \bump(x) \rvert & \leq 2^{k} \cdot \max_{0 \leq k' \leq k} \left[  \sup_{x \in (0,2]} \lvert D^{k'} \bigl( e^{-1/x^m}\bigr) \rvert \right] \cdot  \left[  \sup_{x \in (0,2]} \lvert D^{k - k'} \bigl( e^{-1/x^m}\bigr) \rvert \right] \\
& \leq 2^k  \max_{0 \leq k' \leq k} \left[ 2^{3k'} m^{k'} (k')^{(1+m^{-1})k'} \right] \cdot \left[ 2^{3(k-k')} m^{k-k'} (k-k')^{(1+m^{-1})(k-k')}\right] \\
& \leq (16m)^k \cdot k^{(1+m^{-1}) \cdot k} \qquad\qquad \mbox{for all} \; x \in \R.
\end{align*}
This completes the proof of the lemma.
\end{proof}

We define
\begin{equation}
\label{A_defn}
A(x) := \frac{\pi}{2 \int_\R \bump  dy } \cdot \int_{- \infty}^x \bump(y) dy.
\end{equation}
Since $\bump(y)$ is an even function (see \eqref{c0}), we have 
\begin{equation}
\label{symmetry1}
A(x) + A(-x) = \pi/2.
\end{equation} 
Now let  $\theta(x) := \sin(A(x))$. From \eqref{symmetry1} we deduce that $\theta(-x) = \cos(A(x))$. Therefore,
\begin{equation} \label{c3}
\theta^2(x) + \theta^2(-x) = \sin^2(A(x)) + \cos^2(A(x)) = 1 \quad \mbox{for any} \; x \in \R.
\end{equation}
Since $\bump(x) = 0 $ for $x \in \R \setminus [-1,1]$, we have $A(x) = 0$ for $x \leq -1$, and $A(x) = \pi/2$ for $x \geq 1$. Hence,
\begin{equation}\label{c4}
\theta(x) = 0 \; \mbox{for} \; x \leq -1; \qquad \theta(x) = 1 \; \mbox{for} \; x \geq 1.
\end{equation}

The next lemma provides an estimate on the size of the derivatives of $\theta$. 

\begin{lem}\label{lem2}
Let $F: \C \rightarrow \C$ be an entire function, and let $f : \R \rightarrow \R$ be $C^\infty$.

Assume that there exist $C \geq 1$ and $\gamma \geq 1$ such that $\lvert D^k f(x) \rvert \leq C^k \cdot k^{\gamma k}$ for all $k \geq 0$.

Then there exists $C_0 \geq 1$ determined by $C$, $\gamma$, and $F$, such that $ \lvert  D^k \left[ F(f(x)) \right] \rvert \leq C_0^k \cdot k^{\gamma k}$ for all $k \geq 0$.
\end{lem}

A proof of Lemma \ref{lem2} is given in \cite{DH}.

From the definition \eqref{c0} it is clear that $\int \bump(y) dy \geq \left[ e^{-1}\right]^2 \geq 1/16$. We apply Lemma \ref{der_bounds1} to bound the derivatives of $A(x)$ defined in \eqref{A_defn}. Thus we obtain
\begin{equation*}
\lvert D^k A(x) \rvert \leq  C^k \cdot k^{(1+m^{-1}) \cdot k} \qquad \mbox{for all} \; k \geq 0.
\end{equation*}
We apply Lemma \ref{lem2} to the functions $F(z) = \sin(z)$ and $f(x)= A(x)$. Thus,  for each $m \geq 1$ there exists $C_m \geq 1$ such that
\begin{equation}\label{c5}
\lvert D^k \theta(x) \rvert \leq C_m^k \cdot k^{(1+m^{-1}) \cdot k} \quad \mbox{for all} \; k \geq 0.
\end{equation}

We summarize conclusions \eqref{c3}, \eqref{c4}, and \eqref{c5}, in the following result. 

\begin{prop}\label{cutoff_theorem}
Given a real number $\eta>0$ there exists a $C^\infty$ function $\theta : \R \rightarrow \R$ satisfying
\begin{enumerate}[(a)]
\item $\theta(x) = 0$ for $x \leq -1$, and $\theta(x) = 1$ for $x \geq 1$.
\item $\theta^2(x) + \theta^2(-x) = 1$ for all $x \in \R$.
\item $\lvert D^k \theta(x) \rvert \leq C_\eta^k \cdot k^{(1 + \eta)\cdot k}$ for all integers $k \geq 0$ and all $x \in \R$. Here, $C_\eta \geq 1$ is a constant determined by $\eta$.
\end{enumerate}
\end{prop}

To obtain the previous result from \eqref{c3}, \eqref{c4}, and \eqref{c5}, we choose $m > \eta^{-1}$.

\subsection{Whitney intervals}
\label{whit}

By a \emph{dyadic interval} we mean an interval of the form $(  k \cdot 2^{-\ell}, (k+1) \cdot 2^{-\ell} ]$ for $k, \ell \in \Z$. 

Let $I  = [ - D/2, D/2]$ for $D > 0$. The \emph{Whitney decomposition} of $I$ is a collection $\cW = \{ \cI_j\}_{j \in \cJ} $ consisting of dyadic intervals. Its basic properties are as follows.
\begin{description}
\item[(W1)] The intervals $I_j$ ($j \in \cJ$) are pairwise-disjoint, and $I = \bigcup_{j \in \cJ} I_j$ .
\item[(W2)] We have $\lvert I_j \rvert \leq \dist(I_j,\partial I) \leq 5 \cdot \lvert I_j \rvert$ for all $j \in \cJ$.
\end{description}
For a construction of the Whitney decomposition, see \cite{Stein}.

\subsection{The local cosine basis}

Let $\cW = \{I_j\}_{j \in \cJ}$ be the Whitney decomposition of an interval $I$. We write $I_j = (x_j, x_j + \delta_j]$ for $j \in \cJ$. 

We choose positive real numbers $\eta_j$ and $\eta_j'$ satisfying $\eta_j + \eta_j' \leq \frac{1}{10}\delta_j$. Then define a $C^\infty$ function $\theta_j : \R \rightarrow \R$ by the formula
\begin{equation} \label{cutoff}
\theta_j(x) := \theta\left(\frac{x- x_j}{\eta_{j}} \right) \cdot \theta \left(\frac{(x_j + \delta_j) - x}{\eta_j'}\right).
\end{equation}
Part (a) of Proposition \ref{cutoff_theorem} implies that $\theta_j$ is supported on $\left[x_j- \frac{1}{10} \delta_j, x_j+ \frac{11}{10}\delta_j\right]$. Since $\dist(I_j,\partial  I) \geq | I_j| = \delta_j$ (see \textbf{(W2)}), we have
\begin{equation}
\label{supp_cond}
\supp \theta_j \subset I \;\; \mbox{for all} \; j \in \cJ.
\end{equation}

Denote
\[ \Gamma = \{ (j, k ) : j \in \cJ, \;  k \in \Z, \; k \geq 0 \}.\]
For any $(j, k) \in \Gamma$ we define $ \Phi_{(j,k)} \in C_c^\infty(I)$ by
\begin{equation} \label{lcb} \Phi_{(j,k)}(x) = C_{(j,k)} \cdot \delta_j^{-1/2} \cdot \theta_j(x) \cos\left(\pi \cdot \delta_j^{-1} \cdot \left(k + 1/2\right) \cdot(x - x_j) \right),
\end{equation}
where
\begin{equation}
\label{constdefn}
C_{(j,k)} = \delta_j \cdot \left( \int_\R \theta_j^2(x) \cos^2\left(\pi \cdot \delta_j^{-1} \cdot \left(k + 1/2\right) \cdot(x - x_j) \right) dx \right)^{-1}.
\end{equation}
This definition of $C_{(j,k)}$ ensures the normalization condition $\|\Phi_{(j, k )} \|_{L^2(I)} = 1$. \\
Note that $\theta_j \geq 1$ on $[x_j + \frac{1}{10} \delta_j, x_j + \frac{9}{10}\delta_j]$, and that $\theta_j$ is supported on $[x_j- \frac{1}{10}\delta_j, x_j + \frac{11}{10}\delta_j]$. Therefore, the value of the integral term in the parentheses in \eqref{constdefn} is  between $\frac{1}{100}\delta_j$ and $2\delta_j$. We conclude that
\begin{equation}
\label{Cbd}
\frac{1}{2} \leq C_{(j,k)} \leq 100.
\end{equation}

A theorem of Coifman-Meyer \cite{CM} states that $\{\Phi_{(j,k)}\}_{(j,k) \in \Gamma}$ is an orthonormal basis for $L^2(I)$ for an appropriate choice of the constants $\eta_j$ and $\eta_j'$ which satisfy 
\begin{equation}
\label{eta_bds}
\frac{1}{100}\delta_j \leq \eta_j , \eta_j' \leq \frac{1}{10} \delta_j.
\end{equation} 
This is often called the local cosine basis or the Coifman-Meyer basis. We note that the construction in \cite{CM} uses a different cutoff function $\theta$. However, the proof of orthonormality requires only the conditions found in parts (a) and (b) of Proposition \ref{cutoff}. Consequently, the arguments in \cite{CM} establish orthonormality for the basis constructed here. We note that our cutoff function $\theta$ satisfies the derivative bounds in part (c) of Proposition \ref{cutoff}, which will be used later.

We henceforth assume that $\eta_j$ and $\eta_j'$ satisfy \eqref{eta_bds} and are chosen so that $\{ \Phi_{(j,k)}\}_{(j,k) \in \Gamma}$ is an orthonormal basis for $L^2(I)$.

We write $\widehat{f}$ or $\cF(f)$ to denote the Fourier transform of a function $f \in L^2(\R)$, defined via the formula 
$$\widehat{f}(\xi) = \int_\R f(x) e^{-2 \pi i x \xi} dx.$$

We require the following lemma from \cite{DH}.

\begin{lem}\label{ft_bound}
Let $\theta \in C^\infty(\R)$. Let $C \geq 1$ and $\delta \geq 1$ be such that
\begin{enumerate}[(a)]
\item $\theta$ is supported on $[-1,1]$, and 
\item $\lvert D^k \theta(x) \rvert \leq C^k k^{\delta k}$ for all $k \geq 0$ and $x \in \R$.
\end{enumerate}

Then $\lvert \widehat{\theta} (\xi) \rvert \leq A \exp( - a \cdot | \xi|^{1/\delta}) $ for all $\xi \in \R$, where $a,A > 0$ depend only on $C$ and $\delta$.
\end{lem}

We define $\psi_j(x) := \theta_j( x \cdot \delta_j + x_j)$ for $j \in \cJ$, which can be rewritten as $\psi_j(x) = \theta \left(x \cdot \frac{\delta_j}{\eta_j} \right) \cdot \theta \left( (1-x) \cdot \frac{\delta_j}{\eta_j'} \right)$ (see \eqref{cutoff}). From part (c) of Proposition \ref{cutoff_theorem} and since $\eta_j,\eta_j' \in [\frac{1}{100} \delta_j, \frac{1}{10} \delta_j]$ we conclude that $\lvert D^k \psi_j(x) \rvert \leq C_\eta^k \cdot k^{(1+\eta)\cdot k}$ for $k \geq 0$ and $x \in \R$. By applying Lemma \ref{ft_bound} we learn that
\[
\lvert \widehat{\psi_j}(\xi) \rvert \leq A_{\eta} \cdot \exp \left( - a_\eta \cdot | \xi |^{(1+\eta)^{-1}} \right)
\]
Because of the scaling relationship between $\psi_j$ and $\theta_j$ and simple properties of the Fourier transform, as well as the bound $1 - \eta \leq (1+\eta)^{-1}$ ($\eta>0$), we conclude that
\begin{equation}
\label{expdecay}
\lvert \widehat{\theta}_j(\xi) \rvert \leq A_\eta \cdot \delta_j \cdot \exp( - a_\eta \cdot | \delta_j \cdot \xi |^{1 - \eta} ) \;\; \mbox{for} \; \omega \in \R. 
\end{equation}

Using the formula \eqref{lcb} and the scaling/translation/modulation properties of the Fourier transform $\cF$, we have
\begin{align*}
\cF(\Phi_{(j,k)})(\xi) =  C_{(j,k)} \cdot \delta_j^{-1/2} \cdot \frac{1}{2} \cdot \Biggl[ & \widehat{\theta_j}\left(\xi - \frac{1}{2} (k+1/2) \cdot \delta_j^{-1} \right) \cdot \exp\left(- \pi i \cdot (k+1/2) \cdot \delta_j^{-1} \cdot x_j \right) \\
+ & \widehat{\theta_j} \left(\xi + \frac{1}{2} (k+1/2) \cdot \delta_j^{-1} \right) \cdot \exp \left(\pi i \cdot (k+1/2) \cdot \delta_j^{-1} \cdot x_j \right) \Biggr].
\end{align*}
In particular, thanks to \eqref{expdecay} and \eqref{Cbd} we have
\begin{equation}
\label{uniformbd1}
\lvert \cF(\Phi_{(j,k)}) (\xi) \rvert \leq C_\eta \cdot \delta_j^{1/2} \sum_{\sigma = \pm 1 } \exp \left( - a_\eta \cdot   \left\lvert  \delta_j \cdot \xi - \sigma \cdot \frac{1}{2} \left( k + \frac{1}{2} \right) \right\rvert^{1- \eta} \right),
\end{equation}
for constants $a_\eta > 0$ and $C_\eta > 0$. 

For $k \in \Z_{\geq 0}$ and $j \in \cJ$ we denote
\[
\xi_{jk} = (2k+1) \cdot (4 \delta_j)^{-1}.
\]
If we let $ B_\eta(\xi) := A_\eta \exp \left(- a_\eta \cdot \lvert \xi \rvert^{1 - \eta} \right)$, then the bound \eqref{uniformbd1} states that
\begin{equation}
\label{uniformbd}
\lvert \cF(\Phi_{(j,k)}) (\xi) \rvert \leq \delta_j^{1/2} \cdot \left[ B_\eta\left( \delta_j \cdot ( \xi -  \xi_{jk} ) \right) + B_\eta\left( \delta_j \cdot ( \xi + \xi_{jk} ) \right) \right].
\end{equation}

\section{Energy estimates}\label{sec_energy}
\label{sec3}

Recall that $J = [-1/2,1/2]$ is the frequency localization interval and $I = [-D/2,D/2]$ is the time localization interval. We decompose $I$ into its Whitney decomposition $\cW = \{I_j\}_{j \in \cJ}$. We write $I_j = (x_j,x_j+\delta_j]$.

In the previous section we defined  an orthonormal basis $\{ \Phi_{(j,k)} \}_{j \in \cJ, k \in \Z_{\geq 0} }$ for $L^2(I)$. Recall that $\Phi_{(j,k)}$ is supported on the interval $I_j^* = (x_j - \frac{1}{10} \delta_j, x_j + \frac{11}{10} \delta_j]$. Moreover, the Fourier transform of $\Phi_{(j,k)}$ is (nearly) exponentially concentrated about the frequencies $\xi = \pm \xi_{jk}$ in the sense of the bound \eqref{uniformbd}. In this section we will derive the main energy estimates on our basis.

Let $s \geq 1$ and $\delta_{\min} \in (0,1)$ be parameters, which will be determined in the next section.

We write $X = O(Y)$ to indicate the inequality $\lvert X \rvert \leq C \cdot Y$, where $C$ is a constant independent of all parameters. We write $X = O(Y)$ to indicate the inequality $\lvert X \rvert \leq C_\eta \cdot Y$, where $C$ is a constant depending only on the parameter $\eta$.

Consider the basis functions $\Phi_{(j,k)}$ (${k \in \Z_{\geq 0}}$) associated to a fixed Whitney interval $I_j \in \cW$. We partition this collection into three groups by partitioning the index set $\Z_{\geq 0}$ as follows:
\begin{align}
\label{e2}
& \mathcal{L}^{\low}_j := \{k \in \Z_{\geq 0} : \dist( (2k+1)/(4\delta_j), \R \setminus {J}) \geq s \cdot \delta_j^{-1}  \}, \\
\label{e3}
& \mathcal{L}^{\med}_j := \{  k \in \Z_{\geq 0}  : \dist( (2k+1)/(4\delta_j) , \partial {J}) <  s \cdot \delta_j^{-1} \}, \\
\label{e4}
& \mathcal{L}^{\high}_j := \{k \in \Z_{\geq 0} : \dist((2k+1)/(4\delta_j), {J}) \geq s \cdot \delta_j^{-1} \}.
\end{align}

\begin{lem}\label{local_count}
We have
\begin{itemize}
\item $  \#(\mathcal{L}^{\low}_j)  =   \delta_j + O(s)$ \quad if $\delta_j \geq s, $
\item $ \#(\mathcal{L}^{\low}_j) = 0$ \qquad\qquad\quad if $\delta_j < s, $
\item $\#(\mathcal{L}^{\med}_j) \leq 10 s$.
\end{itemize}
\end{lem} 
\begin{proof}
The numbers $\xi_{jk} = (2k+1)/(4 \delta_j)$ ($k \in \N$) form an evenly-spaced grid of width  $\frac{1}{2 \delta_j}$ starting at $\xi_{j0} = \frac{1}{4 \delta_j}$. Recall that $J = [-1/2,1/2]$. Clearly there are at most $\delta_j + 1$ many indices $k\in\N$ satisfying $\xi_{jk} \in J$. Furthermore, if $k \leq \delta_j - 2 s - 1/2$ then $ \xi_{jk} = \frac{2k+1}{4 \delta_j} \leq \frac{1}{2} - \frac{s}{\delta_j}$,  and thus $\dist(\xi_{jk}, \R \setminus J) \geq s \cdot \delta_j^{-1}$. Therefore, we have
\[
\delta_j - 2s - 1/2 \leq \#(\mathcal{L}_j^{\low}) \leq \delta_j + 1.
\]
This implies the first bullet point.

If $\delta_j < s$ then $s\cdot \delta_j^{-1} > 1$. There are no points whose distance to $\R \setminus [-1/2,1/2]$ is greater than $1$. Thus, in this case $\mathcal{L}_j^{\low} = \emptyset$.

The spacing between consecutive numbers $\xi_{jk}$ is equal to $\frac{1}{2 \delta_j}$. Thus, at most $2s+1$ of the $\xi_{jk}$ lie in an interval of width $s \cdot \delta_j^{-1}$ about the boundary point $1/2 \in \partial J$. The same is true for the boundary point $-1/2 \in \partial J$. Therefore, $\#(\mathcal{L}_j^{\med}) \leq 2 \cdot (2s + 1) \leq 10s$.

\end{proof}

We partition the index set $ \Gamma = \cJ \times \Z_{\geq 0}$ into three components:
\begin{align}
\label{GammaDefn}
& \Gamma_\low = \{ (j,k) :  \delta_j \geq \delta_{\min}, \; k \in \mathcal{L}^{\low}_j\}, \\
\notag{}
& \Gamma_\med =  \{ (j,k) : \delta_j \geq \delta_{\min}, \; k \in \mathcal{L}^{\med}_j \}, \\
\notag{}
& \Gamma_\high =   \{ (j,k) :  \delta_j \geq \delta_{\min}, \; k \in \mathcal{L}^{\high}_j \}  \cup \{  (j,k) :  \delta_j < \delta_{\min}, \;  k \in \Z_{\geq 0} \}.
\end{align}


\begin{lem} \label{count_lem} For a numerical constant $C \geq 0$, we have $\#(\Gamma_{\med}) \leq C s \cdot \log ( D/ \delta_{\min})$.
\end{lem}
\begin{proof}

Lemma \ref{local_count} implies that
\begin{equation*}
\#(\Gamma_{\med}) = \sum_{\substack{j \in \cJ \\ \delta_j \geq \delta_{\min}}} \# (\mathcal{L}^{\med}_j) \leq  \sum_{ \delta_j \geq \delta_{\min}} 10s.
\end{equation*}
Property \textbf{(W2)} implies that the number of Whitney intervals $I_j$ for which $\delta_j \geq \delta_{\min}$ is bounded by \\
$C \log(\diam(I)/\delta_{\min}) = C \log(D / \delta_{\min})$. Hence,
\[ \#(\Gamma_{\med}) \leq C s \cdot \log(D / \delta_{\min}),\] 
as desired.

\comments{Next, Lemma \ref{local_count} implies that
\begin{align*}
\#(\Gamma_{\low}) & = \sum_{\delta_j \geq \delta_{\min}} \# (\mathcal{L}^{\low}_j)  = \sum_{\delta_j \geq s} \#(\mathcal{L}_j^{\low}) \\
& = \sum_{ \delta_j \geq s} \left[ \delta_j  + O(s)\right] \\
& = \sum_{j \in \cJ} \delta_j -  \sum_{\delta_j < s} \delta_j  + O \biggl( \;\; \sum_{\delta_j \geq s} s \biggr).
\end{align*}
We have $\sum_{ j \in \cJ} \delta_j = \diam(I) = D$ because the Whitney intervals $\{I_j\}_{j \in \cJ}$ form a partition of $I$. Moreover, the union of the intervals $I_j$ over all $j$ with $\delta_j < s$ is contained in a neighborhood of the boundary $\partial I$ of radius $C s$, thanks to property \textbf{(W2)}. Since the Whitney intervals are pairwise disjoint, we conclude that
\begin{equation}
\label{whit}
\sum_{\delta_j < s} \delta_j \leq C s.
\end{equation}
Finally, by similar reasoning as before, note that $\sum_{\delta_j \geq s} s \leq Cs \cdot \log(D / s)$. Combining the previous estimates shows that
\[ \#(\Gamma_{\low}) = D + O( s \cdot \log(D/s)),
\]
proving the first bullet point.
}
\end{proof}

We will now prove that Fourier transform of a  basis function $\Phi_{(j,k)}$ indexed by $(j,k) \in \Gamma_\low$ or $(j,k) \in \Gamma_{\high}$ is sharply concentrated on $J$ or $\R \setminus J$, respectively.

\begin{lem}\label{error_lem}
There exist constants $C,c > 0$ determined by $\eta$ such that
\begin{align}
\label{errorbound1}
& \sum_{(j,k) \in \Gamma_{\low}}  \| \cF(\Phi_{(j,k)}) \|^2_{L^2(\R \setminus {J})}  \leq C \exp(-c \cdot s^{1 - \eta}) \cdot \log ( D /\delta_{\min}) \\
\label{errorbound2}
& \sum_{(j,k) \in \Gamma_{\high}}  \| \cF(\Phi_{(j,k)}) \|^2_{L^2( {J})}  \leq C \exp(- c \cdot s^{1 - \eta}) \cdot \log ( D /\delta_{\min}) + C \delta_{\min}.
\end{align}
\end{lem}
\begin{proof}

Recall our notation: $\xi_{jk} = \frac{2k+1}{4 \delta_j}$ for $(j,k) \in \Gamma = \cJ \times \Z_{\geq 0}$.

For $j \in \cJ$ with $\delta_j \geq \delta_{\min}$, we define
\begin{align}
\label{low_piece}
\mathcal{L}^{\low}_{j,\ell} &:= \{  k \in \Z_{\geq 0} : \dist(\xi_{jk}, \R \setminus J) \in [s \cdot 2^\ell/\delta_j, s\cdot 2^{\ell+1} /\delta_j )\},  \quad \mbox{and} \\
\label{high_piece}
\mathcal{L}^{\high}_{j,\ell} &:= \{ k \in \Z_{\geq 0} : \dist( \xi_{jk} , {J}) \in [s \cdot 2^\ell/ \delta_j , s \cdot 2^{\ell+1} / \delta_j) \} \qquad \mbox{for} \; \ell  \in \Z_{\geq 0}.
\end{align}
Note that $\mathcal{L}^{\low}_{j} = \bigcup_{\ell \geq 0} \mathcal{L}^{\low}_{j,\ell}$ and $\mathcal{L}^{\high}_{j} = \bigcup_{\ell \geq 0} \mathcal{L}^{\high}_{j,\ell}$; see \eqref{e2}-\eqref{e4}. 

The spacing between $\xi_{jk}$ and $\xi_{jk'}$ for distinct $k,k' \in \Z_{\geq 0}$ is at least $\frac{1}{2 \delta_j}$. Thus, a counting argument shows that
\begin{equation}
\label{piece_bd1} \#(\mathcal{L}^{\high}_{j,\ell}) \leq  10 s \cdot 2^\ell
\end{equation}
and
\begin{equation}
\label{piece_bd2}
\#(\mathcal{L}^{\low}_{j,\ell}) \leq 10 s \cdot 2^\ell.
\end{equation}

From \eqref{uniformbd}, for any $k \in \Z_{\geq 0}$ we have
\begin{equation*}
\| \cF( \Phi_{(j,k)} ) \|_{L^2(\R \setminus {J})}^2 \leq C \int_{\R \setminus {J}} \delta_j \cdot \bigl[ B_\eta(\delta_j \cdot ( \xi - \xi_{jk})) + B_\eta(\delta_j \cdot ( \xi + \xi_{jk})) \bigr]^2 d \xi.
\end{equation*}
Since $\R \setminus J = (-\infty,1) \cup (1,\infty)$ is symmetric about the origin, we deduce that
\begin{align*}
\| \cF( \Phi_{(j,k)} ) \|_{L^2(\R \setminus {J})}^2 & \leq  C  \int_{\R \setminus J} \delta_j \cdot B_\eta(\delta_j \cdot (\xi - \xi_{jk}))^2 d \xi \\
& \leq C \int_{\lvert \xi' \rvert \geq \delta_j \cdot \dist(\xi_{jk}, \R \setminus {J})} B_\eta( \xi' )^2  d \xi' \notag{}
\end{align*}
where the second inequality relies on the change of variable $\xi' = \delta_j \cdot ( \xi - \xi_{jk})$; note that $\xi \in \R \setminus J \implies |\xi'| \geq \delta_j \cdot \dist(\xi_{jk}, \R \setminus J)$. Thus, by the definition $B_\eta(\xi) = A_\eta \cdot \exp(- a_\eta \cdot | \xi |^{1- \eta})$ we conclude that
\begin{equation}
\label{b1.4}
\| \cF( \Phi_{(j,k)} ) \|_{L^2(\R \setminus {J})}^2 \leq C \exp\left( - c \cdot \left[ \delta_j \cdot \dist(\xi_{jk},\R \setminus {J}) \right]^{1- \eta} \right)
\end{equation}
for constants $c, C > 0$ that depend only on $\eta$. For $k \in \mathcal{L}_{j,\ell}^\low$ we have $\dist(\xi_{jk},\R \setminus J) \sim s \cdot 2^\ell / \delta_j$, where we write $A \sim B$ to mean that $c A \leq B \leq C A$ for some constants $c,C$. Thus, \eqref{b1.4} implies that
\begin{equation}
\label{b1.4new}
\| \cF( \Phi_{(j,k)} ) \|_{L^2(\R \setminus {J})}^2 \leq C \exp\left( - c \cdot \left[ s \cdot 2^\ell \right]^{1- \eta} \right) \;\;\; \mbox{for all} \; k \in \mathcal{L}_{j,\ell}^\low.
\end{equation}

The method used to prove \eqref{b1.4} also shows that
\begin{equation}
\label{b1.5}
\| \cF(\Phi_{(j,k)}) \|_{L^2(J)}^2 \leq C \exp\left( - c \cdot \left[ \delta_j \cdot \dist(\xi_{jk}, {J}) \right]^{1- \eta} \right)
\end{equation}
for constants $c, C > 0$ depending only on $\eta$. For $k \in \mathcal{L}_{j,\ell}^\high$ we have $\dist(\xi_{jk}, J) \sim s \cdot 2^\ell / \delta_j$. Thus, \eqref{b1.5} implies that
\begin{equation}
\label{b1.5new}
\| \cF( \Phi_{(j,k)} ) \|_{L^2(J)}^2 \leq C \exp\left( - c \cdot \left[ s \cdot 2^\ell \right]^{1- \eta} \right) \;\;\; \mbox{for all} \; k \in \mathcal{L}_{j,\ell}^\high.
\end{equation}

We write $\mathcal{L}^{\low}_{j} = \bigcup_{\ell \geq 0} \mathcal{L}^{\low}_{j,\ell}$. Applying \eqref{piece_bd2} and \eqref{b1.4new}, we learn that
\begin{align}
\label{b1} 
\sum_{j : \delta_j \geq \delta_{\min}} \sum_{k \in \mathcal{L}^{\low}_j} \| \cF(\Phi_{(j,k)}) \|^2_{L^2(\R \setminus {J})} &= \sum_{j: \delta_j  \geq \delta_{\min}} \sum_{\ell=0}^\infty \sum_{k \in \mathcal{L}^{\low}_{j,\ell}} \| \cF(\Phi_{(j,k)}) \|^2_{L^2(\R \setminus {J})}\\
\notag{}
& \leq \sum_{j : \delta_j \geq \delta_{\min}} \sum_{\ell=0}^\infty 10 s 2^\ell \cdot C \exp(-c \cdot [s 2^\ell]^{1 - \eta}) \\
\notag{}
& \leq \sum_{ j : \delta_j \geq \delta_{\min}}  C \cdot \exp(-c \cdot s^{1 - \eta}) \\
\notag{}
& \leq C \exp(-c \cdot s^{1 - \eta}) \log(\diam(I)/\delta_{\min}).
\notag{}
\end{align}
In view of the definition of $\Gamma_{\low}$ in \eqref{GammaDefn}, this completes the proof of \eqref{errorbound1}.

Next we prove \eqref{errorbound2}. We write $\mathcal{L}^{\high}_{j} = \bigcup_{\ell \geq 0} \mathcal{L}^{\high}_{j,\ell}$. From \eqref{piece_bd1} and \eqref{b1.5new} we have
\begin{align}
\label{b2}
\sum_{j: \delta_j \geq \delta_{\min}} \sum_{k \in \mathcal{L}^{\high}_{j}} \| \cF(\Phi_{(j,k)}) \|^2_{L^2({J})} &= \sum_{ j : \delta_j \geq \delta_{\min}} \sum_{\ell=0}^\infty \sum_{k \in \mathcal{L}^{\high}_{j,\ell}} \|\cF(\Phi_{(j,k)})\|^2_{L^2({J})}\\
\notag{}
& \leq \sum_{ j  \delta_j \geq \delta_{\min}} \sum_{\ell =0}^\infty 10 s2^\ell \cdot C \exp(-c\cdot [s 2^\ell]^{1 - \eta}) \\
\notag{}
& \leq \sum_{ j : \delta_j \geq \delta_{\min}}  C \cdot \exp(-c \cdot s^{1 - \eta}) \\
\notag{}
& \leq C \exp(-c \cdot s^{1 - \eta}) \log(\diam(I)/\delta_{\min}).
\end{align}

Alternatively, suppose that $j \in \cJ$ is such that $\delta_j < \delta_{\min}$. Then \eqref{uniformbd} implies that
\begin{equation}\label{eq1}
\sum_{k \geq 0} \| \cF(\Phi_{(j,k)})\|^2_{L^2({J})} \leq \sum_{k \geq 0} \int_{{J}} \delta_j \cdot \bigl[ B_\eta(\delta_j \cdot (\xi - \xi_{jk})) +  B_\eta(\delta_j \cdot (\xi + \xi_{jk}))  \bigr]^2 d\xi.
\end{equation}
Switching the order of summation and integration in \eqref{eq1} and using the fact that the interval $J = [-1/2,1/2]$ is symmetric about the origin, we have
\begin{equation}\label{eq2}
\sum_{k \geq 0} \| \cF(\Phi_{(j,k)})\|^2_{L^2({J})} \leq  C \int_{{J}} \sum_{k \geq 0} \delta_j \cdot \bigl[ B_\eta(\delta_j \cdot (\xi - \xi_{jk})) \bigr]^2 d\xi.
\end{equation}
Since $B_\eta(\xi)$ is smooth, bounded, and rapidly decaying as $\xi \rightarrow \infty$, we can compare a Riemann sum with an integral to prove the estimate
\[\sum_{k \geq 0 } \bigl[ B_\eta( \delta_j \cdot ( \xi -\xi_{jk} )) \bigr]^2 \leq  C \int_{\R} B_\eta(z)^2 dz  \leq C \quad \mbox{uniformly for all} \; \xi \in \R.\]
Therefore, from \eqref{eq2} we have
\[
\sum_{k \geq 0 } \| \cF(\Phi_{(j,k)} \|_{L^2({J})}^2 \leq C \delta_j \cdot \lvert {J} \rvert = C \delta_j.
\]
By summing over all $j \in \cJ$ with $\delta_j < \delta_{\min}$, we conclude that
\begin{equation}
\label{b3}
\sum_{ j:  \delta_j < \delta_{\min}} \sum_{k \geq 0} \| \cF(\Phi_{(j,k)} \|_{L^2({J})}^2  \leq C \sum_{j : \delta_j < \delta_{\min}} \delta_j  \leq C \delta_{\min},
\end{equation}
where the last estimate is a consequence of the Whitney conditions \textbf{(W1)} and \textbf{(W2)} (see Section \ref{whit}).

In view of the definition of $\Gamma_{\high}$ in \eqref{GammaDefn}, we see that \eqref{b2} and \eqref{b3} imply the estimate \eqref{errorbound2}, finishing the proof of the lemma.
\end{proof}

\section{Proof of Theorem \ref{thm1}}
\label{sec4}

In the previous section we defined an orthonormal basis $\{ \Phi_{(j,k)}\}_{(j,k) \in \Gamma}$ for $L^2(I)$, depending on a parameter $\eta \in (0,1/2]$.

Let $\epsilon \in (0,1/2)$. Let $s \geq 1$ and $\delta_{\min} \in (0,1)$ be parameters. We will choose $s$ and $\delta_{\min}$ in the following paragraphs.

In the previous section we defined a partition of $\Gamma$ as $\Gamma_\low \cup \Gamma_\med \cup \Gamma_\high$ in terms of the parameters $s$ and $\delta_{\min}$; see \eqref{e2}-\eqref{e4} and \eqref{GammaDefn}.

For all $f \in L^2(I)$, Plancharel's theorem implies that
\begin{align*}
\| T f \|_{L^2(I)}^2 = \| R_I P_J R_I f \|_{L^2(I)}^2 &= \| R_I P_J f \|_{L^2(I)}^2 \\
&\leq \| P_J f \|_{L^2(\R)}^2 \\
&= \int_{J} \lvert \widehat{f}(\xi) \rvert^2 d \xi.
\end{align*}
Similarly,
\begin{align*}
\| f - T f \|_{L^2(I)}^2 = \| f - R_I P_J R_I f \|_{L^2(I)}^2 &= \| R_I f - R_I P_J f \|_{L^2(I)}^2 \\
&\leq \| (I - P_J) f \|_{L^2(\R)}^2 \\
&= \int_{\R \setminus J} \lvert \widehat{f}(\xi) \rvert^2 d \xi.
\end{align*}
Thus
\begin{align}
\label{et1}
\sum_{(j,k) \in \Gamma_\high} \| T \Phi_{(j,k)}\|_{L^2(I)}^2 \;\;+& \sum_{(j,k) \in \Gamma_\low} \| T \Phi_{(j,k)} - \Phi_{(j,k)} \|_{L^2(I)}^2 \\
\notag{}
& \leq \sum_{(j,k) \in \Gamma_\high} \| \cF(\Phi_{(j,k)}) \|_{L^2(J)}^2 + \sum_{(j,k) \in \Gamma_\low} \| \cF(\Phi_{(j,k)})\|_{L^2(\R \setminus J)}^2 \\
\notag{}
& \oleq{\text{Lemma \ref{error_lem}}}  C_\eta  \cdot\exp(- c_\eta \cdot s^{1- \eta}) \cdot \log(D / \delta_{\min}) + C_\eta \cdot \delta_{\min},
\end{align}
where $C_\eta, c_\eta > 0$ are constants determined only by $\eta$.

We are ready, at last, to state our assumptions on $s$ and $\delta_{\min}$. We take
\begin{equation}
\label{set_param}
\left\{
\begin{aligned}
&\delta_{\min} := \frac{\epsilon^3 }{2 C_\eta}, \;\; \mbox{and} \\
&C_\eta \cdot \exp( - c_\eta  \cdot s^{1 - \eta}) \cdot \log(D/\delta_{\min}) \leq \frac{1}{2} \epsilon^3.
\end{aligned}
\right.
\end{equation}
The second estimate  is equivalent to 
\[
s \geq \left( \frac{1}{c_\eta} \log\left( \frac{2 C_\eta \log(D/\delta_{\min})}{\epsilon^3} \right) \right)^{1/(1-\eta)}
\]
Elementary algebra shows that it suffices to take
\begin{equation}
\label{s_defn}
s  := A_\eta \cdot \left( \log \left( \log (D) \cdot \epsilon^{-1} \right) \right)^{1/(1-\eta)},
\end{equation}
for a constant $A_\eta$ determined only by $\eta$. Now using \eqref{set_param} in \eqref{et1} we see that
\[
\sum_{(j,k) \in \Gamma_\high} \| T \Phi_{(j,k)}\|_{L^2(I)}^2 + \sum_{(j,k) \in \Gamma_\low} \| T \Phi_{(j,k)} - \Phi_{(j,k)} \|_{L^2(I)}^2 \leq \epsilon^3.
\]
Recall that $\Gamma$ is equal to the disjoint union $\Gamma_{\low} \cup \Gamma_{\med} \cup \Gamma_{\high}$. Thus, according to  Lemma  \ref{lem1}, if we let $\lambda_k$ ($k \in \N$) denote the positive eigenvalues of $T$ (arranged in non-increasing order), and if we define $ M_\epsilon$ to be the number of eigenvalues of $T$ in the interval $(\epsilon,1-\epsilon)$, then we have
\begin{align*}
M_\epsilon \leq 2 \cdot \# (\Gamma_\med) &\oleq{\text{Lemma \ref{count_lem}}}  C s \cdot \log(D/\delta_{\min}) \\
&= C A_\eta \cdot \left( \log \left( \log (D)  \cdot \epsilon^{-1} \right) \right)^{1/(1-\eta)}  \cdot \log(2 C_\eta D \cdot \epsilon^{-3}) \\
&\leq A_\eta' \cdot \left( \log \left( \log (D)  \cdot \epsilon^{-1} \right) \right)^{1+ 2 \eta}  \cdot \log(D \cdot \epsilon^{-1}),
\end{align*}
for a constant $A_\eta$ determined only by $\eta$. We apply \eqref{pos_mid} and conclude that $\left\{ k : \lambda_k  \in (\epsilon,1-\epsilon) \right\} \subset [ D - 2M_\epsilon , D + 2M_\epsilon]$. According to our upper bound on $M_\epsilon$, this yields
\[
\left\{ k : \lambda_k  \in (\epsilon,1-\epsilon) \right\} \subset [ D - K, D + K],
\]
where
\[
K = A_\eta'' \cdot \left( \log \left( \log (D)  \cdot \epsilon^{-1} \right) \right)^{1 + 2 \eta}  \cdot \log(D \cdot \epsilon^{-1})
\]
 for a constant $A_\eta''$ determined only by $\eta$.
This completes the proof of Theorem \ref{thm1}.
\hfill \qed

\comments{
\begin{proof}[Proof of Theorem \ref{thm1}]

We use the variational approach for estimating the eigenvalues of $T : L^2(I) \rightarrow L^2(I)$. Recall that
\[\lambda_k = \sup_{\substack{ V \subset L^2(I) \\ \dim V = k } } \min_{ \substack{ \psi \in V \\ \| \psi \| = 1 }} \langle T \psi, \psi \rangle .\]

Consider the $N$-dimensional subspace $\overline{V} = \spa\{ \Phi_1,\cdots,\Phi_{N}  \} \subset L^2(J)$. An arbitrary $\psi \in \overline{V}$ with $\| \psi \| = 1$ takes the form
\[ \psi = \sum_{k=1}^{N} c_k \Phi_k, \quad \mbox{where} \; \sum_{k=1}^{N} \lvert c_k \rvert^2 = 1.\] 
From Lemma \ref{main_lem}, we deduce that
\begin{equation}
\label{lbound}
\lambda_{N} \geq \min_{\substack{ \psi \in \overline{V} \\  \| \psi \| = 1} } \langle T \psi, \psi \rangle \geq 1 - C\epsilon.
\end{equation}

At the same time, consider an arbitrary subspace $V \subset L^2(I)$ of dimension $N+M + 1$. Then there exists $\varphi \in V$ such that
\[\varphi = \sum_{k=N+M+1}^\infty c_k \varphi_k,  \quad \mbox{where} \; \sum_{k=N+M+1}^\infty \lvert c_k \rvert^2 = 1.\]
But Lemma \ref{main_lem} shows that  $\langle T \varphi, \varphi \rangle \leq C \epsilon$. Hence,
\begin{equation}
\label{ubound}
\lambda_{N+M+1} = \sup_{\substack{ V \subset L^2(I) \\ \dim(V) = N+M +1 }} \inf_{\substack{\varphi \in V \\ \| \varphi \| = 1 }} \langle T \varphi, \varphi \rangle \leq C \epsilon.
\end{equation}
This complete the proof of Theorem \ref{thm1}.

\end{proof}

}

\end{document}